\documentclass{amsart}
\usepackage{geometry} 
\usepackage{url}               
\usepackage{graphicx}
\usepackage{amssymb}
\usepackage{enumitem}
\usepackage{hyperref}
\usepackage[shortalphabetic]{amsrefs}
\newtheorem{theorem}{Theorem}[section]
\newtheorem{proposition}[theorem]{Proposition}
\newtheorem{conditions}[theorem]{Conditions}
\newtheorem{corollary}[theorem]{Corollary}
\newtheorem{lemma}[theorem]{Lemma}
{\theoremstyle{remark}
\newtheorem{remark}[theorem]{Remark}
\newtheorem*{remark*}{Remark}
\theoremstyle{definition}

\numberwithin{equation}{section}
\newcommand{\p}{\partial}

\newcommand{\RR}{\mathbb R}
\newcommand{\NN}{\mathbb N}
\newcommand{\e}{\epsilon}

\newcommand{\LL}{{\mathcal{L}}}
\newcommand{\h}{{\mathfrak{h}}}
\newcommand{\Acir}{\smash{A\mspace{-8mu}\overset{\circ}{\vphantom{\rule{0em}{1.2ex}}}\mspace{3mu}}}
\DeclareMathOperator{\tr}{tr}

\DeclareMathOperator*{\osc}{osc}
\title[Hypersurfaces with pinched principal curvatures]{Convex hypersurfaces with pinched principal curvatures and flow of convex hypersurfaces by high powers of curvature}
\author{Ben Andrews}
\thanks{Research partially supported by Discovery Grant DP0556211 of the Australian Research Council}
\address{Centre for Mathematics and its Applications, Australian National University, ACT 0200 Australia}
\email{Ben.Andrews@anu.edu.au}
\author{James McCoy}
\address{School of Mathematics and Applied Statistics,
University of Wollongong,
Wollongong, NSW 2522, 
Australia}
\email{jamesm@uow.edu.au}

\begin{document}

\begin{abstract}
We consider convex hypersurfaces for which the ratio of principal curvatures at each point is bounded by a function of the maximum principal curvature with limit $1$ at infinity.  We prove that the ratio of circumradius to inradius is bounded by a function of the circumradius with limit $1$ at zero.   We apply this result to the motion of hypersurfaces by arbitrary speeds which are smooth homogeneous functions of the principal curvatures of degree greater than one.  For smooth, strictly convex initial hypersurfaces with ratio of principal curvatures sufficiently close to one at each point, we prove that solutions remain smooth and strictly convex and become spherical in shape while contracting to points in finite time. 
\end{abstract}
\maketitle

\section{Introduction}

Recently several papers have considered the flow of convex hypersurfaces by speeds which are
homogeneous functions of the principal curvatures of degree $\alpha>1$.  Under suitable pinching conditions on the curvature of the initial hypersurface, the aim is to prove that solutions become spherical as they contract to points.   This behaviour has been established for a wide range of flows where the speed is homogeneous of degree 1 in the principal curvatures, including the mean curvature flow \cite{HuiskenMCF}, the flow by $n$th root of Gauss curvature \cite{ChowGCF}, square root of scalar curvature \cite{ChowSCF}, and a large family of other speeds \cites{AndrewsEuc,AndrewsPinch,Andrews2D}. The first such result with degree of homogeneity higher than $1$ was due to Ben Chow \cite{ChowGCF}, and concerned flow by powers of the Gauss curvature.  He proved that flow by $K^\beta$ with $\beta\geq 1/n$ produces a spherical limiting shape provided the initial hypersurface is sufficiently pinched, in the sense that $h_{ij}\geq C(\beta)Hg_{ij}$.  Later such results were proved by Schulze \cite{Schulze} for powers of the mean curvature, by Alessandroni and Sinestrari for powers of the scalar curvature \cites{Aless,AS}, and by Cabezas-Rivas and Sinestrari \cite{CRS} for normalized flows by powers of elementary symmetric functions.    

A feature of the results mentioned above is that the flows all have some divergence structure, a point which is used crucially in deriving sufficient regularity of solutions to deduce the existence of a smooth limiting hypersurface:  In \cite{ChowGCF} the divergence structure was used to deduce pinching of the principal curvatures using integral estimates in a manner similar to \cite{HuiskenMCF}.  In \cite{Schulze}, \cite{Aless}, and \cite{AS} the curvature pinching was proved using maximum principle arguments, but the divergence structure was still needed because higher regularity of solutions was established using results for divergence form porous-medium equations.   To date we are not aware of any work which provides H\"older continuity for solutions of porous medium equations in non-divergence form without assuming regularity of the coefficients.

In this paper we provide a geometric estimate which circumvents this difficulty:  We prove in Section \ref{sec:geometric} that if the ratios of principal curvatures at each point are bounded in terms of the maximum principal curvature by a function which approaches one at infinity, then the ratio of circumradius to inradius is as close to one as desired if the circumradius is sufficiently small.  Using this, together with a simple argument using spherical barriers, we deduce a positive lower bound on the speed for solutions of the flow equations, at times sufficiently close to the final time.  This enables us to prove a result for flows with no special divergence structure.   Somewhat surprisingly, we are able to prove the result for \emph{arbitrary} smooth speeds which are strictly parabolic and homogeneous of degree $\alpha>1$, provided the pinching ratio of the initial hypersurface is sufficiently close to $1$ depending only on $n$, $\alpha$ and a bound for the second derivatives of the speed.  In particular we make no assumptions involving convexity or concavity of the speed as a function of the principal curvatures.  This is achieved using a parabolic analogue of results of Cordes \cite{Cordes} and Nirenberg \cite{Nirenberg} which give H\"older continuity of first derivatives for solutions of elliptic equations with coefficients close to the identity.

We note that
the first author considered flows of surfaces in three-dimensional space by quite general functions of curvature \cite{Andrews2D}, and the present paper provides the results required to prove smooth convergence to spheres for all the flows considered in that paper.

There are a few results that give a successful treatment of particular high order flows, requiring only (strict) convexity rather than any pinching condition.  The first author proved such a result for Gauss curvature flow of a convex surface in three dimensional space \cite{AndrewsRolling}, and Schulze and Schn\"urer \cite{Schulze} treated flows of surfaces by powers of mean curvature between $1$ and $5$.  Schn\"urer \cite{Schnuerer} treated flow of surfaces with speed $\|A\|^2$, as well as a collection of other particular flows.    While the results of the present paper show that sufficiently strong pinching is preserved by very general flows, understanding the behaviour of such flows without such a pinching requirement seems to be a much more subtle and difficult problem.
 
\section{Notation and preliminary results}

We denote the principal curvatures of a hypersurface $M$ by $\kappa_{\min}=\kappa_1 \leq \ldots \leq \kappa_n = \kappa_{\max}$.  These are the eigenvalues of the Weingarten map, whose matrix is denoted $\mathcal{W} = \left( h^{i}_{\> j} \right)$.  The trace of the Weingarten map is the mean curvature $H$, while $\left\| A \right\|^{2}$ denotes the norm of the second fundamental form $A = \left( h_{ij} \right)$,
$$\left\| A\right\|^{2} = \kappa_{1}^{2} + \ldots + \kappa_{n}^{2} \mbox{.}$$
The trace free second fundamental form $\Acir$ has components $h_{ij} - \frac{1}{n} H g_{ij}$ where $g_{ij}$ denotes the components of the metric on the hypersurface $M$.  The length of $\Acir$ satisfies
\begin{equation}\label{eq:tracelessA}
\left\|\Acir\right\|^{2} = \left\| A \right\|^{2} - \frac{1}{n} H^{2} = \frac{1}{n} \sum_{i<j} \left( \kappa_{i} - \kappa_{j} \right)^{2} \mbox{,}
\end{equation}
which is a measure of the differences between the principal curvatures.  We also make use of the elementary symmetric functions of curvature, defined by
$$
E_k = \frac{1}{\binom{n}{k}}\sum_{1\leq i_1<\dots<i_k\leq n}\kappa_{i_1}\dots\kappa_{i_k}.
$$
   In our analysis of the flow equations we will use several geometric estimates for hypersurfaces.   The first of these appears in \cite{HuiskenMCF}.
   
   \begin{lemma} \label{T:norm}
   $$\left\| \nabla A \right\|^{2} \geq \frac{3}{n+2} \left\| \nabla H \right\|^{2}$$
   and equivalently
    $$\left\|\nabla\Acir\right\|^2=\left\| \nabla_{i}h_{jk} - \frac{1}{n} g_{jk} \nabla_{i}H \right\|^{2} \geq \frac{2\left( n-1\right)}{3n} \left\| \nabla A \right\|^{2} \mbox{.}$$
   \end{lemma}
   
   The next result gives a bound on the ratio of principal curvatures if the length of the traceless second fundamental form is small enough compared to the mean curvature:
      
   \begin{lemma} \label{T:pinch}
  If $H$ is positive and $\left\| \Acir \right\|^{2} \leq \varepsilon H^{2}$ at $p$ with $\e<\frac{1}{n(n-1)}$, then the second fundamental form is positive definite at $p$, and the principal curvatures satisfy
$$
\left(1-\sqrt{n(n-1)\e}\right)\frac{H}{n}\leq \kappa_i\leq \left(1+\sqrt{n(n-1)\e}\right)\frac{H}{n}.
$$
\end{lemma}

\begin{proof} 
Fix $i\in\{1,\dots,n\}$.  Then we have
\begin{align*}
0&\geq \left\| A \right\|^{2}-\left(\frac1n+\e\right)H^{2}\\
&=\kappa_i^2+\sum_{j\neq i}\kappa_j^2-\left(\frac1n+\e\right)H^2\\
&\geq\kappa_i^2+\frac{1}{n-1}\left(\sum_{j\neq i}\kappa_j\right)^2-\left(\frac1n+\e\right)H^2\\
&=\kappa_i^2+\frac{1}{n-1}\left(H-\kappa_i\right)^2-\left(\frac1n+\e\right)H^2\\
&=\frac{n}{n-1}\kappa_i^2-\frac{2}{n-1}\kappa_iH+\left(\frac{1}{n(n-1)}-\e\right)H^2.
\end{align*}
It follows that $z=n\kappa_i/H$ lies between the roots of the equation
$z^2-2z+(1-n(n-1)\e)=0$.
\end{proof}

Define $C= \kappa_{1}^{3} + \ldots + \kappa_{n}^{3}$.  The following estimate generalizes one from \cite{HuiskenVol}.  

\begin{lemma} \label{T:Cest}
If $H>0$ and $\|\Acir\|^2=\e H^2$ at $p$ for some $\e\in\left(0,\frac{1}{n(n-1)}\right)$, then
$$
n C-(1+n\e)H\|A\|^2\geq \e(1+n\e)(1-\sqrt{n(n-1)\e})H^3.
$$
\end{lemma}

\begin{proof}
Let $Q=\|\Acir\|^2-\e H^2=\|A\|^2-(\frac1n+\e)H^2$.  Then $\dot Q^i:=\frac{\partial Q}{\partial\kappa_i}=
2(\kappa_i-(\frac1n+\e)H)$, and $\frac{n}{2}\sum_i\dot Q^i\kappa_i^2 = nC-(1+n\e)H\|A\|^2$.  Noting that by Euler's relation $\sum_i\dot Q^i\kappa_i = 2Q=0$, we write
\begin{align*}
nC-(1+n\e)H\|A\|^2
&=\frac{n}{2}\sum_i\dot Q^i\kappa_i^2\\
&=\frac{n}{2}\sum_i\dot Q^i\kappa_i\left(\left(\kappa_i-\left(\frac1n+\e\right)H\right)+\left(\frac1n+\e\right)H\right)\\
&=n\sum_i\kappa_i\left(\kappa_i-\left(\frac1n+\e\right)H\right)^2+\frac{n}{2}\left(\frac1n+\e\right)H\sum_i\dot Q^i\kappa_i\\
&\geq (1-\sqrt{n(n-1)\e})H\sum_i\left(\kappa_i-\left(\frac1n+\e\right)H\right)^2\\
&=(1-\sqrt{n(n-1)\e})H\left(\|A\|^2-2\left(\frac1n+\e\right)H^2+n\left(\frac1n+\e\right)^2H^2\right)\\
&=(1-\sqrt{n(n-1)\e})\e\left(1+n\e\right)H^3.
\end{align*}
where we used the estimate of Lemma \ref{T:pinch} to obtain the inequality.
\end{proof}

\section{Geometric estimates}\label{sec:geometric}

In \cite{AndrewsEuc}*{Theorem 5.1} it was proved that a pointwise bound on the ratio of principal curvatures of a convex hypersurface implies a bound on the ratio of maximum to minimum width.  
Our main result is a strengthening of that estimate in the case where the ratios of principal curvatures approach one at points where the maximum principal curvature is large.  We denote by $r_-(M)$ the inradius of a convex hypersurface $M$ (the radius of the largest ball enclosed by $M$) and by $r_+(M)$ the circumradius (the radius of the smallest closed ball which contains $M$).

\begin{theorem}\label{thm:geombound}
Let $C_1(\varepsilon)\geq 0$ for each $\varepsilon>0$.  Then for any $\rho>0$ there exists $C_2(\rho)> 0$ such that every smooth convex compact hypersurface $M$ satisfying $\kappa_n\leq \left(1+\varepsilon\right)\kappa_1+C_1(\varepsilon)$ for every $\varepsilon>0$ and $r_+(M)\leq C_2(\rho)$ satisfies $r_+(M)\leq (1+\rho)r_-(M)$.
\end{theorem}

\begin{proof}
Define the (normalized) $k$-dimensional mean cross-sectional volume $V_k(M)$ of $M$ by $V_k(M)=\frac{1}{|S^n|}\int_{M}E_{n-k}\,d\mu(g)$ if $0\leq k\leq n$, and $V_{k}=\frac{1}{|S^n|}\int_{M}s E_{n+1-k}\,d\mu(g)$ for $1\leq k\leq n+1$, where $s=\langle X,\nu\rangle$ is the support function.

The ratio of circumradius to inradius of a convex body can be estimated in terms of the mean cross-sectional volumes.  We will need only the following rather crude statement:

\begin{lemma}\label{lem:radiusratio}
For any $\varepsilon>0$ there exists $\delta(\varepsilon,n)>0$ such that any convex body satisfying
$$
\frac{V_1^{n+1}}{V_{n+1}}\leq 1+\delta
$$
has
$$
\frac{r_+(M)}{r_-(M)}\leq 1+\varepsilon.
$$
\end{lemma}

\begin{proof}
Since both $r_+/r_-$ and $V_1^{n+1}/V_{n+1}$ are scaling invariant, it suffices to assume $V_{n+1}=1$.
The Diskant inequalities \cite{Schneider}*{Theorem 6.2.3} give estimates for the inradius and circumradius:
$$
r_-(M)\geq V_n^{\frac1n}-\left(V_n^{\frac{n+1}{n}}-1\right)^{\frac{1}{n+1}}
\quad\text{and}
\quad
r_+(M)\leq \frac{1}{V_1^{\frac1n}-\left(V_1^{\frac{n+1}{n}}-1\right)^{\frac{1}{n+1}}}.
$$
The function $f(x) = x-\left(x^{n+1}-1\right)^{\frac{1}{n+1}}$ is decreasing in $x$, so the inequality $V_n\leq V_1^n$ implies
$$
r_-(M)\geq V_1-\left(V_1^{n+1}-1\right)^{\frac{1}{n+1}}.
$$
This gives
$$
\frac{r_-(M)}{r_+(M)}\geq \left(V_1-\left(V_1^{n+1}-1\right)^{\frac{1}{n+1}}\right)
\left(V_1^{\frac1n}-\left(V_1^{\frac{n+1}{n}}-1\right)^{\frac{1}{n+1}}\right).
$$
The Lemma follows since the right-hand side approaches $1$ as $V_1$ approaches $1$.
\end{proof}

In view of the Lemma, it suffices to show that the isoperimetric ratio $V_1^{n+1}/V_{n+1}$ can be made close to $1$ by requiring the inradius to be small.  To do this we first observe that the elementary symmetric functions can be compared:  For $1\leq k<\ell\leq n$ we can write 
\begin{equation}\label{eq:EkElineq}
E_k \leq \kappa_n^k\leq \left((1+\varepsilon)\kappa_1+C_1(\varepsilon)\right)^k\leq \left((1+\varepsilon)E_\ell^{1/\ell}+C_1(\varepsilon)\right)^k\leq (1+\varepsilon') E_\ell^{k/\ell}+C_{k,\ell}(\varepsilon')
\end{equation}
for some constant $C_{k,\ell}(\varepsilon')$, for any $\varepsilon'>0$, where we expanded the bracket using the Binomial theorem, used Young's inequality to estimate the resulting terms, and then chose $\varepsilon$ suitably depending on $\varepsilon'>0$.  Integrating \eqref{eq:EkElineq} with $k=n-1$, $\ell=n$ and dividing by $|S^n|$ gives:
\begin{align*}
V_1&\leq \frac{(1+\varepsilon)}{|S^n|}\int_ME_n^{1-1/n}\,d\mu(g)+C_{n-1,n}(\varepsilon)V_n\\
&\leq (1+\varepsilon)\left(\frac{1}{|S^n|}\int_ME_n\,d\mu\right)^{1-1/n} V_n^{1/n}+C_{n-1,n}(\varepsilon)V_n\\
&=(1+\varepsilon)V_n^{1/n}+C_{n-1,n}(\varepsilon)V_n\\
&\leq (1+\varepsilon)V_n^{1/n}\left(1+C_{n-1,n}(\varepsilon)r_+^{n-1}\right).
\end{align*}
Next take \eqref{eq:EkElineq} for $k=1$ and $\ell=n$, multiply by $s/|S^n|$, and integrate over $M$, yielding
\begin{align*}
V_n &=\frac{1}{|S^n|} \int_M E_1s \,d\mu\\
&\leq \frac{1+\varepsilon}{|S^n|}\int_ME_n^{1/n}s \,d\mu(g) + C_{1,n}(\varepsilon)V_{n+1}\\
&\leq (1+\varepsilon)\left(\frac{1}{|S^n|}\int_M E_ns\,d\mu(g)\right)^{1/n}
V_{n+1}^{\frac{n-1}{n}}+C_{1,n}(\varepsilon)V_{n+1}\\
&=(1+\varepsilon)V_1^{\frac1n}V_{n+1}^{\frac{n-1}{n}}+C_{1,n}(\varepsilon)V_{n+1}\\
&\leq (1+\varepsilon)V_1^{\frac1n}V_{n+1}^{\frac{n-1}{n}}\left(1+C_{1,n}(\varepsilon)r_+\right).
\end{align*}
From these two inequalities we have
$$
V_1^{n+1}\leq \left(1+\varepsilon\right)^{\frac{n^2+n}{n-1}}V_{n+1}
\left(1+C(\varepsilon)r_+^{n-1}\right)^{\frac{n^2}{n-1}}\left(1+C(\varepsilon)r_+\right)^{\frac{n^2}{n-1}},
$$
for any $\varepsilon>0$.  Given $\rho>0$, choose $\varepsilon>0$ such that $(1+\varepsilon)^{(n^2+n)/(n-1)}< \sqrt{1+\rho}$, and then choose $C_2(\rho)$ such that $(1+C_{n-1,n}(\varepsilon)C_2^{n-1})(1+C_{1,n}(\varepsilon)C_2)\leq (1+\rho)^{\frac{n-1}{2n^2}}$.  The Theorem follows.
\end{proof}

\begin{remark*}
\begin{enumerate}[label={\arabic*}.]
\item
The argument above also proves a result similar to \cite{AndrewsEuc}*{Theorem 5.1}.   While the precise statement of the result is rather messier than in \cite{AndrewsEuc}, it does give that the ratio of $r_+(M)$ to $r_-(M)$ approaches one if the pointwise ratio of principal curvatures approaches one.  The proof in \cite{AndrewsEuc} gives this only for the ratio of widths.
\item An anisotropic analogue also holds:  If $W$ is a smooth, uniformly convex Wulff shape enclosing the origin, and $M$ is a smooth convex hypersurface, then the \emph{relative normal} $\nu_W:\ M\to W$ takes a point in $M$ to the unique point in $W$ with the same outward normal vector.  The derivative of $\nu_W$ is the \emph{relative Weingarten map}, the eigenvalues $\kappa_1,\dots,\kappa_n$ of which are the \emph{relative principal curvatures}.  The \emph{relative support function} $s_W$ is defined by $X=s_W\nu_W+V$ where $V$ is tangent to $M$, and is given by $s_W = \frac{\langle X,\nu\rangle}{\langle \nu_W,\nu\rangle}$.  The relative $k$-dimensional volumes $V_k(M,W)$ are then defined by exactly the same formulae as $V_k(M)$ above, where $E_k$ is the elementary symmetric function of the relative principal curvatures.  The relative inradius $r_-(M)$ is the supremum of all $r$ such that a translate of $rW$ can be enclosed by $M$, and the relative cirumradius $r_+(M)$ is the infimum of all $r$ such that a translate of $rW$ encloses $M$.  Then Theorem 4 holds as stated (we remark the the Diskant inequality applies unchanged in this situation).
\end{enumerate}
\end{remark*}

\section{The flow equations}
Given a smooth, compact, uniformly convex initial hypersurface $M_{0}^n = X_{0} \left( M \right)$ with $n\geq 2$, we consider a family of embeddings $X: M \times \left[ 0, T\right) \rightarrow \mathbb{R}^{n+1}$ satisfying the evolution equation 
\begin{equation} \label{eq:flow}
  \frac{\partial X}{\partial t} \left( x, t\right) = - F \left( \mathcal{W}\left( x, t\right) \right) \nu\left( x, t\right)
\end{equation}
where $\nu$ denotes the outer unit normal to the evolving hypersurface $M_{t} = X\left( M, t\right) = X_{t}\left( M\right)$, and $F$ is an $O\left(n\right)$-invariant function of ${\mathcal W}$ with the following properties:

\begin{conditions} \label{T:Fconds}
\begin{enumerate}[label={(\roman*).}, ref={(\roman*)}]
  \item\label{item1} $F= f\circ \kappa$, where $\kappa$ is the map which takes a self-adjoint operator to its ordered eigenvalues, and $f$ is a smooth, symmetric function defined on an open symmetric cone $\Gamma\subseteq \Gamma_+$ containing $\{(c,\dots,c):\ c>0\}$, where $\Gamma_+=\{\kappa:\ \kappa_i>0,\ i=1,\dots,n\}$.
    \item\label{item2} $f$ is strictly monotone: $\frac{\partial f}{\partial \kappa_{i}} >0$ for each $i = 1, \ldots, n$, at each point of $\Gamma$.
  \item\label{item3} $f$ is homogeneous of degree $\alpha>1$: $f\left( k \kappa \right) = k^\alpha f\left( \kappa \right)$ for any $k >0$.
  \item\label{item4} $f$ is strictly positive and normalised to have $f\left( 1, \ldots, 1 \right) = n^\alpha$.
\end{enumerate}
\end{conditions}

\begin{remark*}
\begin{enumerate}[label={\arabic*.}, ref={\arabic*}]
\item Since $f$ is a smooth symmetric function by Condition \ref{item1}, it is also a smooth function
of the elementary symmetric functions of the principal curvatures \cite{Glaeser}, so $F$ is a smooth function of
the components of ${\mathcal W}$.
\item The positivity in Condition \ref{item4} follows from homogeneity (Condition \ref{item3}) and strict monotonicity (Condition \ref{item2}) using the Euler relation:  $\alpha F=\sum_i\kappa_i\frac{\partial f}{\partial\kappa_i}$.  The normalization is for convenience, and amounts to scaling the time parameter in equation \eqref{eq:flow}.  The particular choice is made to agree with powers of the mean curvature.
\item\label{Holdercond} Notably absent in our assumptions is any concavity or convexity of $f$.  This is usually indispensible as a condition of this kind is normally required to derive second derivative H\"older estimates for fully nonlinear parabolic equations.  However the situation in the present paper is sufficiently restricted that we can apply a result of Cordes-Nirenberg type which requires no concavity.  The required arguments are provided in Section \ref{sec:Cordes-Nirenberg}.
\end{enumerate}
\end{remark*}

Throughout this paper we sum over repeated indices from $1$ to $n$ unless otherwise indicated.  In computations on the hypersurface $M_{t}$, raised indices indicate contraction with the metric.

We denote by $ \dot{F}^{kl}$ the  derivatives of $F$ with respect to the components of its argument:
$$\left. \frac{\partial}{\partial s} F\left( A+sB\right) \right|_{s=0} = \dot F^{kl} \left( A\right) B_{kl} \mbox{.}$$
Similarly for the second partial derivatives of $F$ we write
$$\left. \frac{\partial^{2}}{\partial s^{2}} F\left( A+sB\right) \right|_{s=0} = \ddot F^{kl, rs} \left( A\right) B_{kl} B_{rs} \mbox{.}$$
If we do not indicate explicitly where derivatives of $F$ and of $f$ are evaluated then they are evaluated at $\mathcal{W}$ and $\kappa\left( \mathcal{W} \right)$ respectively.  We will use similar notation and conventions for other functions of matrices when we differentiate them.

In order to simplify the computations, and since we are not concerned with proving the main result  under the weakest possible pinching condition, we will assume $\Gamma=\{A:\ \|\Acir\|^2<\delta_0 H^2\}$ for some $\delta_0\in\left(0,\frac{1}{n(n-1)}\right)$, and that the second derivatives of $F$ are bounded, in the following sense:  There exists a constant $\mu\geq 0$ such that for any $A$ with $H=\tr A=1$ and $\kappa(A)\in\Gamma$, we have
$\left|\ddot{F}^{kl,rs}(A)B_{kl}B_{rs}\right|\leq\mu\|B\|^2$.
It follows by homogeneity that for arbitrary $A$ with $\kappa(A)\in\Gamma$, 
\begin{equation}\label{eq:D2Fbound}
\left|\ddot{F}^{kl,rs}(A)B_{kl}B_{rs}\right|\leq \mu H^{\alpha-2}\|B\|^2.
\end{equation}
It follows that
\begin{equation}\label{eq:DFbound}
\left(\alpha H^{\alpha-1}-\mu H^{\alpha-2}\|\Acir\|\right) I\leq \dot{F}\leq 
\left(\alpha H^{\alpha-1}+\mu H^{\alpha-2}\|\Acir\|\right) I,
\end{equation}
and that
\begin{equation}\label{eq:Fbound}
H^{\alpha}-\frac{\mu}{2\alpha}H^{\alpha-2}\|\Acir\|^2\leq F\leq H^{\alpha}+\frac{\mu}{2\alpha}H^{\alpha-2}\|\Acir\|^2.
\end{equation}

Our main result regarding the flows is the following:

\begin{theorem}\label{thm:main.flow}
Let $F$ satisfy Conditions \ref{T:Fconds}.  Then there exists $\delta>0$ depending only on $n$, $\alpha$ and $\mu$ such that for any
smooth, uniformly locally convex embedding $X_0:\ M^n\to\RR^{n+1}$ of a compact manifold $M$ satisfying $\|\Acir\|^2<\delta H^2$, there exists a unique solution $X:\ M\times [0,T)\to\RR^{n+1}$ of equation \eqref{eq:flow} with initial data $X_0$, and a point $p\in\RR^{n+1}$ such that $X(M,t)\to p$ as $t\to T$, and 
$\tilde X_t=\frac{X(.,t)-p}{((1+\alpha)(T-t))^{1/(1+\alpha)}}$ converges in $C^\infty$ to an embedding $X_\infty:\ M\to\RR^{n+1}$ with $X_\infty(M)=S^n_1(0)$.
\end{theorem}



Short time existence of a solution to equation \eqref{eq:flow} for smooth, unformly convex initial data is well known (see \cite{AndrewsEuc}, for example).  
The following evolution equations are derived as in \cite{AndrewsEuc}.

\begin{lemma} \label{T:evlnG}
  Under the flow \eqref{eq:flow}, the following evolution equations hold:
  \begin{enumerate}[label={(\roman*)}]
  \item\label{evol.s.f} $\frac{\partial}{\partial t}\langle X,\nu\rangle = \LL\langle X,\nu\rangle+\langle X,\nu\rangle\dot{F}^{kl} h_{km} h^{m}{}_l-(1+\alpha)F$;
    \item  $\frac{\partial}{\partial t}h^{i}_{\> j}=\LL h^{i}{}_{j}  +  \ddot{F}^{kl, rs} \nabla^{i}  h_{k l} \nabla_{j} h_{r s} + \dot{F}^{kl} h_{km} h^{m}{}_l h^{i}{}_j + \left( 1- \alpha \right) F h^{im} h_{mj}$, and
    \item for any smooth symmetric function $G\left( \mathcal{W} \right) = g\left( \kappa\left( \mathcal{W} \right) \right)$,
\begin{multline*}
  \frac{\partial}{\partial t} G = \LL G + \left[ \dot G^{ij} \ddot{F}^{kl, rs} - \dot{F}^{ij} \ddot G^{kl, rs} \right] \nabla_{i} h_{kl} \nabla_{j} h_{rs} + \dot{F}^{kl} h_{km} h^{m}{}_{l} \dot G^{ij}h_{ij}
  + \left( 1 - \alpha \right) F\dot G^{ij} h_{im} h^{m}{}_{j}  \mbox{,}
 \end{multline*}
   where $\LL G = \dot F^{ij} \nabla_{i} \nabla_{j} G$.
 In particular
    \item\label{evolF} $\frac{\partial}{\partial t}F = \LL F +  \dot{F}^{kl} h_{km} h^{m}{}_{l} F$, and 
   \item $\frac{\partial}{\partial t}H = \LL H +  \ddot{F}^{kl, rs} \nabla^{i}  h_{kl} \nabla_{i} h_{rs} +  \dot{F}^{kl} h_{km} h^{m}{}_{l} H + \left( 1 - \alpha \right) F \left\| A \right\|^{2}$.
   \end{enumerate}
   \end{lemma}

\section{Preserving pinching}

In this section we prove the following for solutions of equation \eqref{eq:flow}:

\begin{theorem}\label{thm:pres.pinch}
There exists $\delta_1>0$ depending on $n$, $\mu$ and $\alpha$ such that if $0<\sigma\leq\min\{\delta_0,\delta_1\}$, and $\left\| \Acir \right\|^{2} < \sigma H^2$ at every point at $t=0$, then this remains true as long as the solution exists.
\end{theorem}

\begin{proof}
Using Lemma \ref{T:evlnG} we compute the evolution equation for $Z_\sigma= \left\|\Acir\right\|^2- \sigma H^2$:
\begin{align}\label{eq:dtZsigma}
\frac{\partial Z_\sigma}{\partial t} &= \LL Z_\sigma + 2\! \left(\! h^{ij} \!-\! \left( n^{-1} \!+\! \sigma \!\right)\! H g^{ij}\! \right)\!\! \ddot{F}^{kl, rs} \nabla_{i} h_{kl} \nabla_{j} h_{rs} - 2 \dot{F}^{ij} \!\left( \!\!\nabla_{i} h_{kl} \nabla_{j} h^{kl} - n^{-1} \nabla_{i} H \nabla_{j} H \!\!\right)\notag\\
  &\quad\null+ 2 \sigma \dot{F}^{ij} \nabla_{i} H \nabla_{j} H + 2 \dot{F}^{ij} h_{im} h^{m}_{\ \> j} Z_\sigma + \frac{2}{n} \left( 1 - \alpha \right) F \left( nC - \left( 1 + n \sigma \right) H \left\| A^{2} \right\| \right) \mbox{.}
\end{align}
We argue that $Z_\sigma$ remains negative if it is initially so: Otherwise there exists a first point and time $(x_0,t_0)$ with $Z_\sigma=0$, where $\LL Z_\sigma\leq 0$ and $\partial_tZ_\sigma\geq 0$.  The second-last term in \eqref{eq:dtZsigma} is zero at $(x_0,t_0)$ since $Z_\sigma=0$, and the last is  negative by Lemma \ref{T:Cest}.  It suffices to show that the terms involving $\nabla h$ are non-positive, since then the right-hand side is negative while the left is non-negative.

The $\ddot{F}$ term we estimate using Equation \eqref{eq:D2Fbound}:
\begin{align*}
  2 \left| \left[ h^{ij} - \left(\frac{1}{n}+\sigma \right) H g^{ij} \right] \ddot{F}^{kl, rs} \nabla_{i} h_{kl} \nabla_{j} h_{rs} \right|
 & \leq 2 \left\|h^{ij}-\left(\frac{1}{n}+\sigma \right) H g^{ij}\right\|\left\|\ddot{F}\right\|\|\nabla A\|^2\\
  &\leq 2\sqrt{\sigma(1+n\sigma)}H\left\|\ddot{F}\right\|\|\nabla A\|^2\\
  &\leq 2\sqrt{\sigma(1+n\sigma)}H^{\alpha-1}\mu\|\nabla A\|^2.
\end{align*}
The next term is the good negative term which we estimate as follows:
\begin{align*}
  \dot{F}^{ij} \left[ \nabla_{i} h_{kl} \nabla_{j} h^{kl} - \frac{1}{n} \nabla_{i} H \nabla_{j} H \right]
  &=  \dot{F}^{ij}  \left( \nabla_{i} h_{kl} - \frac{1}{n} g_{kl} \nabla_{i}H \right) \left( \nabla_{j} h^{kl} - \frac{1}{n} g^{kl} \nabla_{j} H \right)\\
  &\geq \left(\alpha H^{\alpha-1}-\mu H^{\alpha-2}\|\Acir\|\right)\left\|\nabla \Acir\right\|^2\\
  &\geq \frac{2\left( n-1\right)}{3n}H^{\alpha-1}\left(\alpha -\mu \sqrt{\sigma}\right)\left\|\nabla A\right\|^2 \mbox{,}
\end{align*}
where we used Equation \eqref{eq:DFbound} and Lemma \ref{T:norm}.
The same estimates also control the remaining term:
\begin{align*}
\left|2 \sigma \dot{F}^{ij} \nabla_{i} H \nabla_{j} H\right|
&\leq 2\sigma\left(\alpha+\mu\sqrt{\sigma}\right)H^{\alpha-1}\|\nabla H\|^2\\
&\leq \frac{2(n+2)}{3}\sigma\left(\alpha+\mu\sqrt{\sigma}\right)H^{\alpha-1}\|\nabla A\|^2.
\end{align*}

Putting these estimates together, the gradient terms at $(p_0,t_0)$ in \eqref{eq:dtZsigma} are no greater than
$$
\left(-\frac{4(n-1)}{3n}(\alpha-\mu\sqrt{\sigma})+2\sqrt{\sigma(1+n\sigma)}\mu
+\frac{2(n+2)}{3}\sigma\left(\alpha+\mu\sqrt{\sigma}\right)\right)H^{\alpha-1}\|\nabla A\|^2.
$$
The bracket is clearly nonpositive on $0\leq \sigma\leq \delta_1$ for some $\delta_1>0$ depending on $n$, $\mu$ and $\alpha$.\end{proof}

\begin{remark}
By Lemma \ref{T:evlnG} \ref{evolF}, the minimum of $F$ does not decrease in $t$.  Since $(n\kappa_{\max})^\alpha=f(\kappa_{\max},\dots,\kappa_{\max})\geq f(\kappa_1,\dots,\kappa_n)$, $\kappa_{\max}$ remains strictly positive and $H > 0$ by Lemma \ref{T:pinch}.
\end{remark}

\begin{remark} We used the second derivative bound $\mu$, but no concavity of $f$.  One-sided second derivative bounds might suffice, but mere parabolicity does not (unless $n=2$ as shown in \cite{Andrews2D}) --- for example no pinching cones exist if $F=(H+c\kappa_2)^\alpha$ with $c>3$ for $n=3$.
\end{remark}

\section{Curvature bound}

Next we prove upper bounds on the speed, by using an argument of Tso \cite{Tso} to show that the speed remains bounded if the inradius stays positive.
Suppose $r_-(M_t)\geq r_0$ for $0\leq t\leq t_0$.  Take the origin to be the centre of the insphere of $M_{t_0}$, so that $M_t$ encloses $B_{r_0}(0)$ for $0\leq t\leq t_0$.  
Let $Q=\frac{F}{2\langle X,\nu\rangle-r_0}$.  By Theorem \ref{thm:pres.pinch} and Lemma \ref{T:pinch}, $\kappa_{\min}\geq \frac{1-\sqrt{n(n-1)\sigma}}{1+\sqrt{n(n-1)\sigma}}\kappa_{\max}$ with $\sigma<\min\{\delta_0,\delta_1\}$, so
$$
\dot{F}^{kl}h_k{}^ph_{pl}\geq \alpha \kappa_{\min} F\geq \frac{\alpha(1-\sqrt{n(n-1)\sigma})}{n(1+\sqrt{n(n-1)\sigma})}F^{1+1/\alpha},
$$
where we observed $F=f(\kappa_1,\dots,\kappa_n)\leq f(\kappa_{\max},\dots,\kappa_{\max})=\left(n\kappa_{\max}\right)^\alpha$.  Using this we compute
\begin{align*}
\frac{\partial Q}{\partial t}
&= \LL Q + 4\dot{F}^{kl}\frac{\nabla_k\langle x,\nu\rangle}{2\langle X,\nu\rangle-r_0}\nabla_lQ
+\frac{F}{(2\langle X,\nu\rangle-t_0)^2}\left(r_0\dot{F}^{kl}h_k{}^ph_{pl}-(1+\alpha)F\right)\\
&\leq \LL Q+ 4\dot{F}^{kl}\frac{\nabla_k\langle X,\nu\rangle}{2\langle X,\nu\rangle-r_0}\nabla_lQ
+(1+\alpha)Q^2 - r_0^{1+1/\alpha}\tilde C Q^{2+1/\alpha},
\end{align*}
where $\tilde C=
\frac{\alpha(1-\sqrt{n(n-1)\sigma})}{n(1+\sqrt{n(n-1)\sigma}}$, and we used $2\langle X,\nu\rangle-r_0\geq r_0$.  The maximum principle then implies
$$
Q\leq \max\left\{\left(\frac{2(1+\alpha)}{\tilde C}\right)^\alpha r_0^{-(1+\alpha)},
\left(\frac{(1+\alpha)\tilde C}{2\alpha}\right)^{-\alpha/(1+\alpha)}r_0^{-1}t^{-\alpha/(1+\alpha)}\right\}.
$$
It follows that $F$ is bounded, and hence so is $\kappa_{\max}$.  We have proved the following:

\begin{proposition}\label{prop:curv.bound}
Suppose $X:\ M\times [0,t_0]\to\RR^{n+1}$ is a solution of \eqref{eq:flow} with $\|\Acir\|^2<\min\{\delta_0,\delta_1\}H^2$ at $t=0$.  Then there exists $C_+$ depending only on $n$, $\alpha$, $\mu$, $r_+(M_0)$ and $r_-(M_{t_0})$ such that $\kappa_{\max}(x,t)\leq C_+\left(1+t^{-1/(1+\alpha)}\right)$
for all $(x,t)\in M\times(0,t_0]$.
\end{proposition}
 
\section{Regularity results for parabolic equations}\label{sec:Cordes-Nirenberg}

Here we summarise the ingredients needed for our regularity results. Denote by $Q_r$ the spacetime cylinder $B_r(0)\times(-r^2,0]$.  We need the result on H\"older continuity of solutions of linear uniformly parabolic equations by Krylov and Safonov \cite{KS} (see also \cite{Lieb:book}*{Corollary 7.26}):  

\begin{theorem}\label{thm:KS}
For any $0<\lambda\leq\Lambda<\infty$ and $B\geq 0$, there exist $\gamma\in(0,1)$ and $K$ such that for any smooth solution $u:\ Q_1\to\RR$ of the equation 
\begin{equation}\label{eq:unif.par.lin}
\frac{\p u}{\p t} = a^{ij}\frac{\p^2 u}{\p x^i\p x^j} + b^i\frac{\p u}{\p x^i} + f,
\end{equation}
with $\lambda \|\xi\|^2\leq a^{ij}\xi_i\xi_j\leq \Lambda\|\xi\|^2$ for all $\xi$ and $|b^i|\leq B$, the following holds for $0<r<1$:
$$
\osc_{Q_r}u\leq K r^\gamma\left(\osc_{Q_1}u+\|f\|_\infty\right).
$$
\end{theorem}

Schauder estimates give more regularity if the coefficients are regular (see \cite{Lieb:book}*{Theorem 4.9}):

\begin{theorem}\label{thm:Schauder}
Let $0<\lambda\leq\Lambda<\infty$, $A,B,C\geq 0$, and $\beta\in(0,1)$.  Then there exists $K$ depending only on $n$, $\lambda$, $\Lambda$, $A$, $B$, $C$ and $\beta$ such that for any
smooth solution $u:\ Q_1\to\RR$ of the equation
\begin{equation}\label{eq:unif.par.lin.2}
\frac{\p u}{\p t} = a^{ij}\frac{\p^2 u}{\p x^i\p x^j} + b^i\frac{\p u}{\p x^i} + cu + f
\end{equation}
where $\lambda\|\xi\|^2\leq a^{ij}\xi_i\xi_j\leq \Lambda\|\xi\|^2$ for all $\xi$, $|a^{ij}(x,t)-a^{ij}(y,s)|\leq A\max\{|x-y|^\beta,|t-s|^{\beta/2}\}$, $|b^i(x,t)|\leq B$, $|b^i(x,t)-b^i(y,s)|\leq B\max\{|x-y|^\beta,|t-s|^{\beta/2}\}$, $|c(x,t)|\leq C$ and $|c(x,t)-c(y,s)|\leq C\max\{|x-y|^\beta,|t-s|^{\beta/2}\}$ for all $(x,t)$ and $(y,s)$ in $B_1(0)\times[-1,0]$, the following estimate holds:
\begin{align*}
&\|Du\|_{L^\infty(Q_{1/2})}
+\|D^2u\|_{L^\infty(Q_{1/2})}
+\left\|\frac{\p u}{\p t}\right\|_{L^\infty(Q_{1/2})}
+\sup_{(x,t)\neq (y,s)\text{\ in\ }Q_{1/2}}\frac{\left|\frac{\p u}{\p t}(x,s)-\frac{\p u}{\p t}(y,s)\right|}{\max\{|x-y|^{\beta},|t-s|^{\beta/2}\}}\\
&\null+\sup_{(x,t)\neq (y,s)\text{\ in\ }Q_{1/2}}\frac{\left|D^2u(x,t)-D^2u(y,s)\right|}{\max\{|x-y|^{\beta},|t-s|^{\beta/2}\}}
+\sup_{(x,t),(x,s)\in Q_{1/2},\ t\neq s}\frac{\left|Du(x,t)-Du(x,s)\right|}{|t-s|^{(1+\beta)/2}}\\
&\qquad\leq K\left(\|u\|_{L^\infty(Q_1)}+\|f\|_{L^\infty(Q_1)}+\sup_{(x,t)\neq(y,s)\text{\ in\ }Q_1}
\frac{\left|f(x,t)-f(y,s)\right|}{\max\{|x-y|^{\beta},|t-s|^{\beta/2}\}}\right).
\end{align*}
\end{theorem}

In order to apply Schauder estimates to deduce higher regularity for solutions of fully nonlinear parabolic equations, one must first prove that derivatives up to second order in space and first order in time are H\"older continuous.  In most such cases it is necessary to use the results of Krylov \cite{Krylov} (or Evans \cite{Evans} for the elliptic case) which require concave dependence on the second derivatives.    However in the situation of this paper the geometric restrictions allow us to dispense with this requirement:  Cordes \cite{Cordes} and Nirenberg \cite{Nirenberg} proved that linear elliptic equations with coefficients close to the identity admit $C^{1,\alpha}$ estimates.  We use a result of this kind from \cite{Lieb:book}*{Lemma 12.13}, proved by adapting to the parabolic setting a method of Caffarelli \cite{Caffarelli}.

\begin{theorem}\label{thm:cordes}
Let $0<\lambda\leq\Lambda<\infty$, $A>0$ and $B>0$.  Then there exist constants $C>0$, $\gamma\in(0,1]$ and $\sigma>0$ such that for any $u:\ Q_1\to\RR$ with continuous first time derivatives and second spatial derivatives, satisfying the equation
\begin{equation}\label{eq:par.eqn}
\frac{\partial u}{\partial t} = a^{ij}(x,t)D_iD_ju + a(x,t)
\end{equation}
with $\lambda\delta^{ij}\leq a^{ij}(x,t)\leq\Lambda\delta^{ij}$, $|Du|\leq B$, $|a(x,t)|\leq A$, and $|a^{ij}(x,t)-a^{ij}(y,s)|\leq\sigma$ for all $(x,t)$ and $(y,s)$ in $Q_1$, we have for all $0<r<1$
$$
\osc_{Q_r}Du\leq Cr^\gamma\left(
\osc_{Q_1}Du+K\right).
$$
\end{theorem}

\section{H\"older continuity of second derivatives}\label{sec:Holderbounds}

In this section we use the regularity results provided in the previous section to prove H\"older continuity of the second fundamental form.  The proof proceeds in several stages:  First, we observe that upper and lower bounds on $F$ together with the pinching result of Theorem \ref{thm:pres.pinch} imply bounds above and below on second derivatives, which in turn imply that the evolving hypersurfaces can be described locally as the graph of a function evolving by a uniformly parabolic equation.  From this we deduce that the unit normal $\nu$ and the speed $F$ are H\"older continuous in space-time.  It follows that on small spacetime regions the evolving graphs have first spatial derivatives satisfying an equation to which Theorem \ref{thm:cordes} can be applied, yielding the following:

\begin{theorem}\label{thm:h-Holder}
Let $F$ satisfy Conditions \ref{T:Fconds}.  Then there exist $\delta_2\in(0,\delta_1]$ and $\gamma\in(0,1]$ depending only on $n$, $\alpha$ and $\mu$ such that for any smooth uniformly convex solution $X:\ M\times [0,\tau]\to\RR^{n+1}$ of \eqref{eq:flow} with $\|\Acir\|^2<\min\{\delta_2,\delta_0\} H^2$ everywhere, 
$$
\left|h_{x,t}(v,v)-h_{y,s}(w,w)\right|\leq C\left(|X(x,t)-X(y,s)|^\gamma+|t-s|^{\gamma/2}+\left|X_*(v)-X_*(w)\right|^{\gamma}\right)
$$
for all $(x,t),(y,s)\in M\times [\tau/2,\tau]$ with $|X(x,t)-X(y,s)|+\sqrt{|t-s|}<d$, $v\in T_xM$, and $w\in T_yM$, where $C$ and $d$ depend on $\tau$, $r_+=r_+(M_0)$, $r_-=r_-(M_\tau)$ and $\inf_{M_0}F$.
\end{theorem}

\begin{proof}
We first deduce a curvature bound:

\begin{lemma}\label{lem:curv-bounds}
For $X$ as in Theorem \ref{thm:h-Holder} there exist  $0<k_-<k_+<\infty$ depending on $\tau$, $n$, $\mu$, $\alpha$, $r_\pm$ and $\inf_{M_0}F$ such that $k_-\leq\kappa_i(x,t)\leq k_+$ for all $x\in M$, $t\in[\tau/8,\tau]$ and $i\in\{1,\dots,n\}$.
\end{lemma}

\begin{proof}
The uniform convexity of $M_0$ implies that the speed $F$ has a positive lower bound at $t=0$, and the evolution equation for $F$ in Lemma \ref{T:evlnG}\ref{evolF} implies that this lower bound is maintained as long as the solution exists.  By Proposition \ref{prop:curv.bound} we deduce an upper bound on $F$ on $M\times[\tau/8,\tau]$, depending only on $r_-(M_\tau)$ and $r_+(M_0)$.  
Theorem \ref{thm:pres.pinch} and Lemma \ref{T:pinch} imply upper and lower bounds on the principal curvatures.
\end{proof}

In order to apply the results of the previous section we write the evolving hypersurface locally as an evolving graph:  Choose $x_1\in M$ and $t_1\in[\tau/8,\tau]$.  Since $\kappa_i\leq k_+$, $M_{t_1}$ encloses a sphere of radius $1/k_+$ which touches at $X(x_1,t_1)$.  Choose the origin to be at the centre of this sphere.  Then $M_t$ lies outside the sphere of radius $\left(k_+^{-(1+\alpha)}-(1+\alpha)(t-t_1)\right)^{\frac{1}{1+\alpha}}$ for $t_1\leq t<t_1+\frac{1}{(1+\alpha)k_+^{1+\alpha}}$.  In particular $M_t$ lies outside the sphere of radius $2^{-\frac{1}{1+\alpha}}k_+^{-1}$ for $t\in\left[t_1,t_1+\frac{1}{2(1+\alpha)}k_+^{-(1+\alpha)}\right]$.  Choose an orthonormal basis for $\RR^{n+1}$ with $e_{n+1}=k_+X(x_1,t_1)$.  By convexity, we have
$$
M_t\cap\{X\cdot e_{n+1}>0\}\cap\left\{|X-\left(e_{n+1}\cdot X\right)e_{n+1}|<2^{-\frac{1}{1+\alpha}}k_+^{-1}\right\}
=\left\{(z,-u(z,t)):\ |z|<2^{-\frac{1}{1+\alpha}}k_+^{-1}\right\}
$$ 
where $u\in C^\infty\left(B_{2^{-\frac{1}{1+\alpha}}k_+^{-1}}(0)\times\left[t_1,t_1+\frac{1}{2(1+\alpha)}k_+^{-(1+\alpha)}\right]\right)$ is convex for each $t$, $0\geq u(z,t)\geq -k_+^{-1}$ and $|Du(z,t)|\leq 2^{1+\frac{1}{1+\alpha}}$ for $|z|\leq 2^{-\frac{1}{1+\alpha}}k_+^{-1}$.  We bound the second spatial derivatives $D^2u$:  The second fundamental form is $h_{ij} = \frac{D_iD_ju}{\sqrt{1+|Du|^2}}$, and the induced metric is $g_{ij}=\delta_{ij}+D_iu\,D_ju$.  The upper and lower bounds on principal curvatures give
$k_-g_{ij}\leq h_{ij}\leq k_+g_{ij}$, which becomes
$$
k_- \left(\delta_{ij}+D_iu\,D_ju\right)\sqrt{1+|Du|^2}\leq D_iD_ju\leq k_+ \left(\delta_{ij}+D_iu\,D_ju\right)\sqrt{1+|Du|^2},
$$
so that by the gradient bounds we have 
\begin{equation}\label{eq:D2ubound}
k_-\delta_{ij}\leq D_iD_ju\leq k_+\left(1+2^{2+\frac{2}{2+\alpha}}\right)^{\frac32}.
\end{equation}

The function $u$ evolves according to the scalar parabolic evolution equation
\begin{equation}\label{eq:graph.flow}
\frac{\partial u}{\partial t} = \left(1+|Du|^2\right)^{\frac{1-\alpha}{2}}F\left(P\circ D^2u\circ P\right)=:\hat F(D^2u,Du),
\end{equation}
where $P_{ij}=\delta_{ij}-\frac{D_iu D_ju}{\sqrt{1+|Du|^2}\left(1+\sqrt{1+|Du|^2}\right)}$ is the square root of the inverse of the induced metric.  

Our next step is to prove that the normal direction is H\"older-continuous (in space-time):

\begin{lemma}\label{lem:normal-Holder}
For $X$ as in Theorem \ref{thm:h-Holder}, there exists $C$ and $d$ depending only on $r_+(M_0)$, $r_-(M_\tau)$ and $\inf_{M_0}F$ such that
$$
\left|\nu(x,t)-\nu(y,s)\right|\leq C\left(|X(x,t)-X(y,s)|+\sqrt{|t-s|}\right)
$$
for all $(x,t)$ and $(y,s)$ in $M\times[\tau/8,\tau]$ with $|X(x,t)-X(y,s)|+\sqrt{|t-s|}<d$.
\end{lemma}

\begin{proof}
 Differentiating Equation \eqref{eq:graph.flow} in a spatial direction gives the following evolution equation for $v=\frac{\partial u}{\partial z^i}$ for any $i$:
\begin{equation}\label{eq:graph.flow.ut}
\frac{\partial v}{\partial t}
= a^{kl}D_kD_lv+b^kD_kv
\end{equation}
where $a^{kl}=\frac{\partial\hat F(r,p)}{\partial r_{kl}}\Big|_{(r,p)=(D^2u,Du)}$ and 
$b^k=\frac{\partial\hat F(r,p)}{\partial p_{k}}\Big|_{(r,p)=(D^2u,Du)}$.   Explicitly we have
\begin{equation}\label{eq:a-form}
a^{kl} =\sqrt{1+|Du|^2}\dot F^{pq}\Big|_{P^T(D^2u)P} P^{pk}P^{ql},
\end{equation}
and $b^k$ is an expression involving $Du$ and $D^2u$, which may be bounded as follows: 
\begin{equation}\label{eq:b-bound}
|b^k|\leq C(|Du|,|D^2u|)|Du|.
\end{equation}
In particular $|b^k|\leq B$, where $B$ depends only on $r_+(M_0)$, $r_-(M_\tau)$ and $\inf_{M_0}F$.  We note that by equations \eqref{eq:a-form} and \eqref{eq:DFbound} and \eqref{eq:Fbound},
there exists $\Lambda>0$ depending on $\mu$, $n$, $\alpha$ and $r_\pm$ such that $a^{ij}\xi_i\xi_j\leq \Lambda\|\xi\|^2$ for all $\xi$.  

Equation \ref{eq:D2ubound} provides Lipschitz dependence of $v$ in space.  We use a barrier argument to deduce the required continuity in time:  Since $v(0,t_1)=0$, and $|v(x,t)|\leq 2^{1+\frac{1}{1+\alpha}}$, and $|Dv(x,t)|\leq K=k_+\left(1+2^{2+\frac{2}{1+\alpha}}\right)^{\frac32}$ everywhere on $B_{2^{-\frac{1}{1+\alpha}}k_+^{-1}}(0)\times\left[t_1,t_1+\frac{1}{2(1+\alpha)}k_+^{-(1+\alpha)}\right]$, we have in particular that $|v(x,t)|\leq K|x|\leq \frac{K^2}{4\varepsilon}+\varepsilon|x|^2$ for any $\varepsilon>0$, for $|x|\leq 2^{-\frac{1}{1+\alpha}}k_+^{-1}$ when $t=t_1$ and for $|x|=2^{-\frac{1}{1+\alpha}}k_+^{-1}$ when $t_1\leq t\leq t_1+\frac{1}{2(1+\alpha)}k_+^{-(1+\alpha)}$.  Sub and super solutions for Equation \eqref{eq:graph.flow} are then given by $v^{(\varepsilon)}_\pm = \pm\left(\frac{K^2}{4\varepsilon}+\varepsilon|x|^2+2\varepsilon\left(\Lambda n+2^{-\frac{1}{1+\alpha}}k_+^{-1}B\right)(t-t_1)\right)$, and so by the comparison principle we have
$
|v(x_1,t)|\leq \frac{K^2}{4\varepsilon}+2\varepsilon(\Lambda n+2^{-\frac{1}{1+\alpha}}k_+^{-1}B)(t-t_1).
$
Optimizing over $\varepsilon$ for each $t$ gives $|v(x_1,t)|\leq \sqrt{2(\Lambda n+2^{-\frac{1}{1+\alpha}}k_+^{-1}B)K^2(t-t_1)}$.  The continuity of the unit normal $\nu$ follows since it is a smooth function of the gradient.  
\end{proof}

We can now deduce H\"older continuity of the speed $F$:

\begin{lemma}\label{lem:F-Holder}
For $X$ as in Theorem \ref{thm:h-Holder}, there exist $\gamma\in(0,1)$, $C$ and $d$ depending only on $r_+(M_0)$, $r_-(M_\tau)$ and $\inf_{M_0}F$ such that 
$$
\left|F(x,t)-F(y,s)\right|\leq C\left(|X(x,t)-X(y,s)|^\gamma+|t-s|^{\gamma/2}\right)
$$
for all $(x,t)$ and $(y,s)$ in $M\times [\tau/4,\tau]$ with $|X(x,t)-X(y,s)|+\sqrt{|t-s|}<d$.
\end{lemma}

\begin{proof}
As before we write the evolving hypersurface locally as an evolving graph satisfying Equation \eqref{eq:graph.flow}.  Differentiating with respect to time, we find that the function $v=\frac{\partial u}{\partial t}$ again satisfies equation \eqref{eq:graph.flow.ut}.  The expression \eqref{eq:a-form} for $a^{ij}$, together with bounds on $Du$ and $F$ and equations \eqref{eq:DFbound} and \eqref{eq:Fbound}, imply that there exist $\lambda$ and $\Lambda$ such that $\lambda\|\xi\|^2\leq a^{ij}\xi_i\xi_j\leq\Lambda\|\xi\|^2$ for all $\xi$.  The bounds on $Du$ and $D^2u$ also imply that and $|b^i|\leq B$.   The function $\tilde v(z,t) = R^{-1}v(Rz,t_1+(1+t)R^2t)$ with $R=\min\left\{1,2^{-\frac{1}{1+\alpha}}k_+^{-1},\frac{1}{\sqrt{2(1+\alpha)}}k_+^{-\frac{1+\alpha}{2}}\right\}$ gives a solution of the equation \eqref{eq:graph.flow.ut} on $Q_1$ with $b$ replaced by $Rb$ (hence $|b|\leq B$) still holds), and so Theorem \ref{thm:KS} applies.  Rescaling back we have the estimate
$$
\osc_{Q_{r}(0,t_1+R^2)} v \leq K\left(\frac{r}{R}\right)^\gamma\osc_{Q_R(0,t_1+R^2)}v
$$
for $0<r<R$.  Finally, since $\frac{\partial u}{\partial t}$ and $Du$ are H\"older continuous, so is $F$ by 
Equation \eqref{eq:graph.flow}.
\end{proof}

Now we complete the proof of Theorem \ref{thm:h-Holder}:  Choose any $(x_1,t_1)$ with $t_1\in[\tau/4.\tau]$.  We first rescale to ensure that $F(x_1,t_1)=1$, by setting $Y(x,t)=F(x_1,t_1)^{{1}/{\alpha}}X\left(x,t_1+F(x_1,t_1)^{-\frac{1+\alpha}{\alpha}}t\right)$ and noting that $Y$ is again a solution of \eqref{eq:flow}.
We write the evolving hypersurface $Y_t(M)$ locally near $(x_1,0)$ as a graph evolving according to Equation \eqref{eq:graph.flow}, and let $v=\frac{\partial u}{\partial z^i}$ for some $i\in\{1,\dots,n\}$, so that $v$ evolves according to Equation \eqref{eq:graph.flow.ut}.  
The expressions \eqref{eq:a-form} and \eqref{eq:b-bound}, together with the H\"older continuity of $Du$ and $F$, imply that on small enough regions we can make $|Du|$ as small as desired (in particular less than $1$) and $F$ as close as desired to $1$.  By the expression \eqref{eq:a-form} for $a^{ij}$, and the bounds \eqref{eq:DFbound} and \eqref{eq:Fbound}, on small regions we have $\lambda\|\xi\|^2\leq a^{ij}\xi_i\xi_j\leq\Lambda\|\xi\|^2$, where $\lambda$ and $\Lambda$ depend only on $n$, $\alpha$ and $\mu$.  Similarly, since $F$ is close to $1$ we have bounds above and below on principal curvatures from Lemma \ref{T:pinch}, and hence bounds $|Dv|\leq B$ and $|b^iD_iv|\leq A$, where $A$ depends only on $n$, $\alpha$ and $\mu$.  Let $\sigma$ be given by Theorem \ref{thm:cordes} with these values of $\lambda$, $\Lambda$, $A$ and $B$.
Finally, by \eqref{eq:DFbound} and \eqref{eq:Fbound} we can choose $\delta_2$ sufficiently small (depending on $n$, $\alpha$, $\mu$ and $\sigma$) to ensure that $(\alpha-\sigma/4)I\leq \dot{F}(A)\leq (\alpha+\sigma/4)I$ whenever $F(A)=1$ and $\|\Acir\|^2\leq\delta_2 H^2$.  It follows that on small enough regions we have $|a^{ij}(x,t)-a^{ij}(y,s)|<\sigma$.

Now let $\tilde v(z,t)=R^{-1}v(Rz,(1+t)R^2t)$ for $R$ to be chosen.  Then $v$ satisfies Equation \eqref{eq:graph.flow.ut} on $Q_1$ (for $R$ sufficiently small) with $b^i$ replaced by $Rb^i$, and hence for sufficiently small $R$ (depending on $n$, $\alpha$, $\mu$, $r_{\pm}$ and $\inf_{M_0}F$) we can apply Theorem \ref{thm:cordes} to deduce $C^{1,\gamma}$ estimates on $\tilde v$.  Translating back to $v$ we deduce $C^{2,\gamma}$ bounds on $u$.  This implies H\"older continuity of the second fundamental form of the solution $Y_t$, and again scaling back gives H\"older continuity of the second fundamental form of $X_t$, as claimed.
\end{proof}

\section{Higher regularity}

The higher regularity for solutions of equation \eqref{eq:flow} now follows from Schauder estimates:

\begin{proposition}\label{prop:high.reg}
For each $k\in\NN$ there exists $C_k$ depending on $t_0$, $n$, $\alpha$, $\mu$, $r_{\pm}$ and $\inf_{M_0}F$ such that for any solution $X:\ M\times[0,t_0]\to\RR^{n+1}$ of equation \eqref{eq:flow} satisfying $\|\Acir\|^2<\min\{\delta_0,\delta_2\}H^2$, $\|\nabla^kA\|\leq C_k$ on $M\times[t_0/2,t_0]$.
\end{proposition}

\begin{proof}
We again work with the local graph parametrization used in the previous section, around some point $(x_1,t_1)$.
Since $\nabla^kA$ is controlled once the derivatives up to order $k$ are controlled when the hypersurface is written locally as a graph, it suffices to obtain bounds on $D^ku$ for the corresponding solutions of Equation \eqref{eq:graph.flow}.

By Theorem \ref{thm:h-Holder}, the function $u$ is $C^{2,\gamma}$, with H\"older bounds on second derivatives depending on $n$, $\alpha$, $\mu$, $r_{\pm}$, $\inf_{M_0}F$, and elapsed time.  In particular this implies that the coefficients of Equation \eqref{eq:graph.flow.ut} are H\"older continuous, and therefore by Theorem \ref{thm:Schauder} we deduce $C^{2,\gamma}$ bounds on $v=\frac{\partial u}{\partial z^i}$, and hence $C^{3,\gamma}$ bounds on $u$. The higher regularity follows by induction, by considering the evolution equations satisfied by higher spatial derivatives of $u$:  H\"older continuity of $m$th spatial derivatives of $u$ allow Theorem \ref{thm:Schauder} to be applied to the evolution equation for order $(m-1)$ spatial derivatives, yielding H\"older continuity of order $m+1$ derivatives.  By induction this yields bounds on all spatial derivatives, and by the evolution equation this also implies bounds on arbitrary spatial and time derivatives of $u$ also.
\end{proof}

\section{Convergence to a point}

\begin{proposition}\label{prop:point}
Let $M_0=X_0(M)$ be a smooth, uniformly convex, compact hypersurface with $\|\Acir\|^2<\min\{\delta_0,\delta_2\} H^2$.  Then the maximally defined solution $X:\ M\times[0,T)\to\RR^{n+1}$ of equation \eqref{eq:flow} with initial data $X_0$ is $C^\infty$, each hypersurface $M_t=X_t(M)$ is strictly locally convex, and there exists $p\in\RR^{n+1}$ such that $X_t(M)\to p\in\RR^{n+1}$ uniformly as $t\to T$.
\end{proposition}

\begin{proof}
Suppose the inradius does not approach zero as the final time is approached.
At $t=0$ the speed $F$ has a positive lower bound, since $M_0$ is uniformly convex, and the evolution equation for $F$ in Lemma \ref{T:evlnG}\ref{evolF} implies that this lower bound is maintained as long as the solution exists.  Proposition \ref{prop:curv.bound} implies that $F$ is bounded since the inradius does not approach zero, and the pinching estimate of Theorem \ref{thm:pres.pinch} and Lemma \ref{T:pinch} then imply bounds above and below on principal curvatures.  
Proposition \ref{prop:high.reg} then provides bounds on all higher derivatives of curvature.  In this case the solution can be continued to a larger time interval, by the argument of \cite{HuiskenMCF}*{Theorem 8.1}.  This would contradict the maximality of the interval of existence, so the inradius must approach zero as $t\to T$.  The circumradius does also, by \cite{AndrewsEuc}*{Lemma 5.4}.
\end{proof}

\section{Improving pinching}

Now we prove that the bounds on the pinching ratio can be improved to show that the pinching ratio is close to one when the curvature becomes large.  For $\alpha>1$ this can be done using the maximum principle, rather than the more involved integral estimates and iteration techniques which are required to prove the corresponding estimates for mean curvature flow in \cite{HuiskenMCF}.

The result is the following:

\begin{theorem}\label{thm:imp.pinch}
Let $\delta_1>0$ be as in Theorem \ref{thm:pres.pinch}.  Let $M_{0} = X_{0} \left( M \right)$ be a smooth, compact, uniformly convex hypersurface for which $\|\Acir\|^2<\sigma_0H^2$ with $\sigma_0\in(0,\min\{\delta_1,\delta_0\})$.  Let $\h=\sup_{M_0}H$, and let $X: M \times \left[ 0, T\right) \rightarrow \mathbb{R}^{n+1}$ be the solution of \eqref{eq:flow} with initial data $X_0$.  Then there exists $\lambda>0$ such that for any $(x,t)\in M\times[0,T)$ the second fundamental form $A$ of $M_t$ at $x$ satisfies
$\|\Acir\|^2\leq\min\left\{\sigma_0 H^2,\sigma_0\h^\lambda H^{2-\lambda}\right\}$.
\end{theorem}

\begin{proof}
At $t=0$ we have $H\leq \h$, and so $\|\Acir\|^2<\sigma_0 H^2=\min\left\{\sigma_0 H^2,\sigma_0\h^\lambda H^{2-\lambda}\right\}$.  
We have already proved that the condition $\|\Acir\|^2\leq \sigma_0H^2$ is preserved.  It remains to prove that for suitable $\lambda>0$ the quantity $Z=\|\Acir\|^2-\sigma_0 \h^\lambda H^{2-\lambda}$ cannot attain a new zero maximum at a point where $H\geq \h$.  At such a point we let $\sigma=\sigma_0\left(\frac{\h}{H}\right)^\lambda\leq\sigma_0$, and note that $Z=0$, $\LL Z\leq 0$ and $\partial_tZ\geq 0$.  The latter are related by the evolution equation for $Z$:
\begin{align}
\frac{\partial Z}{\partial t} &= \LL Z + 2\!\left(\!h^{ij} - \!\left(\!\frac{1}{n} + \sigma\!\right)\!H g^{ij} \!\right)\! \ddot{F}^{kl, rs} \nabla_{i} h_{kl} \nabla_{j} h_{rs} - 2 \dot{F}^{ij} \!\left(\!\!\nabla_{i} h_{kl} \nabla_{j} h^{kl}\!-\! \frac{1}{n} \nabla_{i} H \nabla_{j} H\!\!\right)\notag\\
  &\quad\null+ 2 \sigma \dot{F}^{ij} \nabla_{i} H \nabla_{j} H +\lambda\sigma H\ddot{F}(\nabla_ih,\nabla_ih)-\sigma\lambda(3\!-\!\lambda\!)\dot{F}^{ij}\nabla_iH\nabla_jH+ 2 \dot{F}^{ij} h_{im} h^{m}{}_{j} Z\label{eq:dtZ}\\
  &\quad\null+\lambda\sigma  \dot{F}^{ij} h_{im} h^{m}{}_{j}H^2 \!-\!\frac{2(\alpha-1)F}{n}\!\!\!\left(\!nC \!-\! \left( 1\! + \!n \sigma \right)\! H \!\left\| A^{2} \right\|\!+\!\frac{n\lambda\sigma}{2}H\|A\|^2\! \!\right)\!\!.\notag
\end{align}
The terms which do not involve $\lambda$ can be estimated exactly as in the proof of Theorem \ref{thm:pres.pinch}.  For the remaining terms we have the following straightforward estimates:  Choosing $\lambda\leq 3$ we discard the third term on the second line.  The preceding term we estimate by
$$
\lambda\sigma H\ddot{F}(\nabla_ih,\nabla_ih)\leq \lambda\sigma\mu H^{\alpha-1}\|\nabla A\|^2.
$$
Thus all of the terms involving gradients of curvature can be estimated by $H^{\alpha-1}\|\nabla A\|^2$ times
$$
-\frac{4(n-1)}{3n}(\alpha-\mu\sqrt{\sigma})+2\sqrt{\sigma(1+n\sigma)}\mu
+\frac{2(n+2)}{3}\sigma\left(\alpha+\mu\sqrt{\sigma}\right)+\lambda\sigma\mu.
$$
Since $\sigma\leq \sigma_0<\delta_1$, the first three terms have a strictly negative sum, and so the result is nonpositive provided $\lambda$ is sufficiently small depending on $\alpha$, $n$, $\mu$ and $\sigma_0$.

The terms in the last line of equation \eqref{eq:dtZ} can be estimated as follows:  We discard the negative term in the last bracket, and estimate the others using Lemma \ref{T:Cest}.  To bound the preceding term we note $h_{im}h^m{}_j\leq Hh_{ij}$ since the principal curvatures are non-negative, and so
$$
\lambda\sigma \dot{F}^{ij}h_{im}h^m{}_j H^2\leq\lambda\sigma\dot F^{ij}h_{ij}H^3 = \lambda\alpha\sigma FH^3.
$$
This gives the following bound for the last line of
\eqref{eq:dtZ}:
$$
\left(\lambda\sigma - \frac{2}{n}(\alpha-1)\sigma(1+n\sigma)\left(1-\sqrt{n(n-1)\sigma}\right)\right)FH^3,
$$
which is strictly negative for $\lambda>0$ sufficiently small (depending only on $n$, $\alpha$, $\mu$ and $\sigma_0$).
\end{proof}

\begin{corollary}\label{cor:weak.pinch}
If $X$ is a solution of equation \eqref{eq:flow} as in Theorem \ref{thm:imp.pinch}, then for any $\varepsilon>0$ there exists a constant $C_1(\varepsilon)\geq 0$ such that $\kappa_{n}\leq (1+\varepsilon)\kappa_1+C_1(\varepsilon)$.
\end{corollary}

\begin{proof}
By Theorem \ref{thm:imp.pinch} we have
\begin{align*}
\kappa_n-\kappa_1&\leq \sqrt{n}\|\Acir\|\\
&\leq \sqrt{n}\sigma_0\h\left(1+(H/\h)^{1-\lambda/2}\right)\\
&\leq \sqrt{n}\sigma_0\h\left(1+\frac{1}{p}a^{-p}+\frac{1}{q}a^q(H/\h)^{q(1-\lambda/2)}\right)
\end{align*}
for any $a>0$ and $1/p+1/q=1$, by Young's inequality.  Choosing $q(1-\lambda/2)=1$ gives 
$$
\kappa_n-\kappa_1\leq \sqrt{n}\sigma_0a^{2/(2-\lambda)}H+\sigma_0\h(1+\lambda/2a^{-2/\lambda})
\leq \varepsilon\kappa_1+C_1(\varepsilon),
$$
where 
$$
C_1(\varepsilon)= \sigma_0\h\left(1+\frac{\lambda}{2}\left(\frac{\varepsilon(1-\sqrt{n(n-1)\sigma_0})}{n^{\frac32}\sigma_0}\right)^{-\frac{2-\lambda}{\lambda}}\right), 
$$
and we chose $a^{2/(2-\lambda)} = \varepsilon(1-\sqrt{n(n-1)\sigma_0})/(n^{3/2}\sigma_0)$ and used Lemma \ref{T:pinch}.
\end{proof}

\section{lower speed bounds}

The results of Theorem \ref{thm:geombound} and Corollary \ref{cor:weak.pinch} imply that the evolving hypersurfaces are close to spheres when their diameter is small.  In this section we use this to deduce lower bounds on the speed.  Our argument here is quite general and does not require any assumptions on the speed except for homogeneity and parabolicity.

\begin{proposition}\label{prop:speedbound}
Let $X: M\times [0,T)\to\RR^{n+1}$ be a smooth, strictly convex solution of a flow of the form \eqref{eq:flow} which contracts to a point as $t\to T$, and such that the ratio of circumradius to inradius approaches $1$ as $t\to T$.   Then there exists $t_-<T$ and $C>0$ such that $F(z,t)\geq C(T-t)^{-\frac{\alpha}{1+\alpha}}$ for all $z\in \bar M$ and $t\in[t_-,T)$.
\end{proposition}

\begin{proof}

We begin with the following observation, originally due to Smoczyk \cite{Smoczyk}*{Proposition 4} for the mean curvature flow.

\begin{lemma}\label{lem:Smoc}
If $M_{t_0}$ encloses $p\in\RR^{n+1}$, then $\langle x-p,\nu\rangle +(1+\alpha)(t-t_0)F\geq 0$ for $t_0\leq t<T$.
\end{lemma}

\begin{proof}[Proof of Lemma] 
This can be proved directly from the evolution equations in Lemma \ref{T:evlnG}, \ref{evol.s.f} and\ref{evolF}.  We give a proof which gives some further insight in to the meaning of the result:

The homogeneity of $F$ implies that 
$x_\lambda(z,t)=\lambda \left(x(z,\lambda^{-(1+\alpha)}(t-t_0))-p\right)+p$ is again a solution of the same flow, for any $\lambda>0$ and $p\in\RR^{n+1}$.
When $t=t_0$ this amounts to a rescaling about the point $p$, so if $M_{t_0}$ encloses $p$ then $x_\lambda(\bar M,t_0)$ encloses $M_{t_0}$ for any $\lambda>1$.  By the comparison principle, $x_\lambda(\bar M,t)$ encloses $M(t)$ for each $t\geq t_0$ and $\lambda>1$, and we deduce that 
$$
\left\langle \frac{\partial}{\partial\lambda}x_\lambda(z,t)\Big|_{\lambda=1},\nu(z,t)\right\rangle \geq 0
$$
for all $z$ and $t\geq t_0$.  Computing the left-hand side explicitly, we find
$$
0\leq \left\langle x-p+(1+\alpha)(t-t_0)F\nu,\nu\right\rangle =\langle x-p,\nu\rangle + (1+\alpha)(t-t_0)F.
$$
\end{proof}

\noindent\emph{Proof of Proposition \ref{prop:speedbound}, Continued.} 
 Choose $t_1$ close enough to $T$ to ensure $r_+(M_t)\leq (1+\varepsilon)r_-(M_t)$ for all $t\in[t_1,T)$, where $\varepsilon>0$ is such that $(1+2\varepsilon)^{1+\alpha}<5/4$.  Fix $t_0\in[t_1,T)$, and write $r_-=r_-(M_{t_0})$.   Choose $q\in\RR^{n+1}$ to be the incentre of $M_{t_0}$, so that 
$B_{r_-}(q)$ is enclosed by $M_{t_0}$, and $B_{(1+2\varepsilon)r_-}(q)$ encloses $M_{t_0}$.  By the comparison principle, $B_{r(t)}(q)$ is enclosed by $M_t$ for $t\geq t_0$ where 
$r(t)=\left(r_-^{1+\alpha}-(1+\alpha)(t-t_0)\right)^{1/(1+\alpha)}$, and $B_{R(t)}(q)$ encloses $M_t$ for $t\geq t_0$ where $R(t)=\left(((1+2\varepsilon)r_-)^{1+\alpha}-(1+\alpha)(t-t_0)\right)^{1/(1+\alpha)}$.
In particular this implies $T\geq t_0+\frac{r_-^{1+\alpha}}{1+\alpha}$.

Fix $x\in M$, and let $t=t_0+\frac{3r_-^{1+\alpha}}{4(1+\alpha)}$.  Choose $q$ to be the point in $M_{t_0}$ which maximizes $\langle q,\nu(x,t)\rangle$, and apply Lemma \ref{lem:Smoc}.  This gives
$$
F(x,t)\geq \frac{\langle q-X(x,t),\nu(x,t)\rangle}{(1+\alpha)(t-t_0)}
\geq \frac{4(r_--R(t))}{3r_-^{1+\alpha}}.
$$
The estimates above give $R(t)=r_-\left((1+2\varepsilon)^{1+\alpha}-3/4\right)^{1/(1+\alpha)}
\leq 2^{-1/(1+\alpha)}r_-$ by our choice of $\varepsilon$.  Also note that $T-t\geq \frac{r_-^{1+\alpha}}{4(1+\alpha)}$.  These imply
$$
F(x,t)\geq  \frac{4(1-2^{-1/(1+\alpha)})}{3(4(1+\alpha)(T-t))^{\alpha/(1+\alpha)}} = C(T-t)^{-\frac{\alpha}{1+\alpha}}
$$
as required, for all $x\in M$ and all $t\geq t_-=t_1+\frac{3r_-(t_1)^{1+\alpha}}{4(1+\alpha)}$.
\end{proof}

\begin{remark*}
An inspection of the proof will reveal that the result of Proposition \ref{prop:speedbound} holds provided the ratio of circumradius to inradius is less than some explicit constant depending on $\alpha$ and $n$, so the assumption that this ratio approaches $1$ could be weakened.
\end{remark*}

\section{The convergence theorem}

\begin{proof}[Proof of Theorem \ref{thm:main.flow}]

The existence until the solution converges to a point $p$ was proved in Proposition \ref{prop:point}.  
We next prove Hausdorff convergence of the rescaled hypersurfaces $\tilde M_t=\tilde X_t(M)$ to the unit sphere:  
By Corollary \ref{cor:weak.pinch} and Theorem \ref{thm:geombound}, for any $\rho>0$ there exists $C_2(\rho)$ such that $r_+(M_t)\leq (1+\rho)r_-(M_t)$ whenever $r_+(M)\leq C_2(\rho)$.  Since $M_t$ converges to a point we have $r_+(M_t)\to 0$ as $t\to T$, so for each $\rho>0$ there exists $t(\rho)<T$ such that $r_+(M_t)\leq (1+\rho)r_-(M_t)$ for $t(\rho)\leq t<T$, and since $M_t$ converges to $p$, $p$ is enclosed by $M_t$ for each $t$.  We can relate $r_\pm(M_t)$ to $T-t$:  The time of existence is no less than that of the sphere of radius $r_-(M_t)$, and no greater than that of the sphere of radius $r_+(M_t)$, so 
\begin{equation}\label{eq:Tvsr}
\frac{r_-(M_t)^{1+\alpha}}{1+\alpha}\leq T-t\leq \frac{r_+(M_t)^{1+\alpha}}{1+\alpha}\leq \frac{((1+\rho)r_-(M_t))^{1+\alpha}}{1+\alpha}
\end{equation}
for $t\geq t(\rho)$.  We can also control the distance from $p$ to the centre $p_t$ of the insphere of $M_t$:  $p$ is enclosed by $M_{t'}$, which is enclosed by the sphere of radius $\left(r_+(M_t)^{1+\alpha}-(1+\alpha)(t'-t)\right)^{1/(1+\alpha)}$ about $p_t$ for $t\leq t'<T$.  Taking $t'\to T$ gives $|p-p_t|\leq  
(\left(r_+(M_t)^{1+\alpha}-(1+\alpha)(T-t)\right)^{1/(1+\alpha)}\leq
\left(r_+(M_t)^{1+\alpha}-r_-(M_t)^{1+\alpha}\right)^{1/(1+\alpha)}$.  This implies that
\begin{equation}\label{eq:ptvsp}
\frac{|p_t-p|}{((1+\alpha)(T-t))^{1/(1+\alpha)}}\leq \left((1+\rho)^{1+\alpha}-1\right)^{1/(1+\alpha)}
\end{equation}
for $t>t(\rho)$.  The estimates \eqref{eq:Tvsr} and \eqref{eq:ptvsp} amount to Hausdorff convergence of $\tilde M_t$ to the unit sphere.

Next we observe that by Proposition \ref{prop:speedbound}, for $t\geq t_-$ we have 
$F\geq C(T-t)^{-\frac{\alpha}{1+\alpha}}$, and hence by the pinching estimate of Theorem \ref{thm:pres.pinch} and Lemma \ref{T:pinch} we have $\kappa_{\min}\geq (T-t)^{-\frac{1}{1+\alpha}}$.

We deduce an upper curvature bound from Proposition \ref{prop:curv.bound}, using scaling to deduce the correct dependence on the remaining time:   Fix $t_*\in[T/2,T)$, and consider rescaled solutions defined by  
$$
\tilde X^{(t_*)}(p,t) = \frac{X(p,t_*+(T-t_*)t)}{(T-t_*)^{\frac{1}{1+\alpha}}}.  
$$
For any $t^*\in [T/2,T)$, $\tilde X^{(t_*)}$ is a solution of equation \eqref{eq:flow} on $M\times [-1,0]$.  Given $\rho>0$, if $T-t_*<(T-t(\rho))/2$ then by \eqref{eq:Tvsr} we have $r_-(\tilde X^{(t_*)}_t)\geq \frac{(1+\alpha)^{\frac{1}{1+\alpha}}}{1+\rho}$ and $r_+(\tilde X^{(t_*)}_t)\leq (1+\rho)(2(1+\alpha))^{\frac{1}{1+\alpha}}$, for $-1\leq t\leq 0$.  It follows from Proposition \ref{prop:curv.bound} that the rescaled solutions have $F(x,0)\leq C$ independent of $t_*$, and hence the unrescaled solutions have
$F\leq C(T-t)^{-\frac{\alpha}{1+\alpha}}$.

Estimates on higher derivatives of curvature also follow:  The solutions $\tilde X^{(t_*)}$ defined above have bounds on curvature in $C^k$ for any  $k$ on $M\times [-1/2,0]$, by Proposition \ref{prop:high.reg}.
It follows that $\|\nabla^{(k)}A\|\leq C_k(T-t)^{-\frac{\alpha(1+k)}{1+\alpha}}$ for each $k$, and the rescaled hypersurfaces $\frac{M_t-p}{((1+\alpha)(T-t))^{\frac{1}{1+\alpha}}}$ have uniform bounds on  curvature and all higher derivatives.  

$C^\infty$ convergence of the rescaled embeddings now follows:   The pinching estimate of Theorem 
\ref{thm:imp.pinch} implies that (for $t$ close enough to $T$) the rescaled maps satisfy $\|\Acir\|^2\leq C(T-t)^{\frac{\lambda}{1+\alpha}}$.  By interpolation we also have $\|\nabla\Acir\|\leq C(T-t)^{\beta}$ for some $\beta>0$, and hence by Lemma \ref{T:norm} $\|\nabla A\|\leq C(T-t)^{\beta}$.  Then interpolation again gives that $\|\nabla^{(k)}A\|\leq C_k(T-t)^{\beta_k}$ for every $k$, where $\beta_k>0$.  One can define a normalized flow as in \cite{HuiskenMCF}*{Section 9}.  The time variable in the normalized flow is $\tau=\log(T-t)$, so we have exponential decay of all derivatives of the second fundamental form for the normalized flow, and the argument in \cite{HuiskenMCF}*{Section 10} can be applied.
\end{proof}

\begin{remark*}
A similar result to Theorem \ref{thm:main.flow} also holds in the case $\alpha=1$:   By a slight modification of the proof of Theorem \ref{thm:pres.pinch}, flow by an essentially arbitrary smooth, homogeneous degree one speed $F$ preserves sufficiently strong pinching of a convex hypersurface, with pinching ratio determined by $n$ and a bound for the second derivatives of $F$.  The argument given in Section \ref{sec:Holderbounds} applies with minor changes to prove H\"older continuity of the second fundamental form, and higher regularity follows by Schauder estimates.  Convergence of the rescaled hypersurfaces then follows by the argument given in \cite{AndrewsEuc}.  We note that a lower bound on $F$ follows in this case from the Krylov-Safonov Harnack inequality, so this situation does not need the geometric result of Section \ref{sec:geometric}.
\end{remark*}


\begin{bibdiv}
\begin{biblist}
\bib{Aless}{thesis}{
	author={Alessandroni, Roberta},
	title={Evolution of hypersurfaces by curvature functions},
	date={2008},
	eprint={http://dspace.uniroma2.it/dspace/handle/2108/661},
	organization={Universit\'a degli studi di Roma ``Tor Vergata''},
	type={Ph.D. Thesis},
}	

\bib{AS}{article}{
    author={Alessandroni, Roberta},
    author={Sinestrari, Carlo},
    title={Evolution of hypersurfaces by powers of the scalar curvature},
    status={to appear in Ann. Sc. Norm. Super. Pisa Cl. Sci.},
    date={2009},
}
    
\bib{AndrewsEuc}{article}{
   author={Andrews, Ben},
   title={Contraction of convex hypersurfaces in Euclidean space},
   journal={Calc. Var. Partial Differential Equations},
   volume={2},
   date={1994},
   number={2},
   pages={151--171},
   issn={0944-2669},
   review={\MR{1385524 (97b:53012)}},
}


\bib{AndrewsRolling}{article}{
   author={Andrews, Ben},
   title={Gauss curvature flow: the fate of the rolling stones},
   journal={Invent. Math.},
   volume={138},
   date={1999},
   number={1},
   pages={151--161},
   issn={0020-9910},
   review={\MR{1714339 (2000i:53097)}},
}

\bib{AndrewsPinch}{article}{
   author={Andrews, Ben},
   title={Pinching estimates and motion of hypersurfaces by curvature
   functions},
   journal={J. Reine Angew. Math.},
   volume={608},
   date={2007},
   pages={17--33},
   issn={0075-4102},
   review={\MR{2339467 (2008i:53087)}},
}

\bib{Andrews2D}{article}{
   author={Andrews, Ben},
   title={Deforming surfaces by non-concave curvature functions},
   eprint={arXiv:math.DG/0402273}
}

\bib{AndrewsReg}{article}{
	author={Andrews, Ben},
	title={Fully nonlinear parabolic equations in two space variables},
	eprint={arXiv:math.AP/0402235}
}
	
\bib{CRS}{article}{
    author={Cabezas-Rivas, Esther},
    author={Sinestrari, Carlo},
    title={Volume-preserving flow by powers of the $m$th mean curvature},
    eprint={	arXiv:0902.2090v1},
    status={to appear in Calc. Var. PDE},
    date={2009},
}

\bib{Caffarelli}{article}{
   label={Ca},
   author={Caffarelli, Luis A.},
   title={Interior a priori estimates for solutions of fully nonlinear
   equations},
   journal={Ann. of Math. (2)},
   volume={130},
   date={1989},
   number={1},
   pages={189--213},
   issn={0003-486X},
   review={\MR{1005611 (90i:35046)}},
}

\bib{ChowGCF}{article}{
   label={Ch1},
   author={Chow, Bennett},
   title={Deforming convex hypersurfaces by the $n$th root of the Gaussian
   curvature},
   journal={J. Differential Geom.},
   volume={22},
   date={1985},
   number={1},
   pages={117--138},
   issn={0022-040X},
   review={\MR{826427 (87f:58155)}},
}

\bib{ChowSCF}{article}{
   label={Ch2},
   author={Chow, Bennett},
   title={Deforming convex hypersurfaces by the square root of the scalar
   curvature},
   journal={Invent. Math.},
   volume={87},
   date={1987},
   number={1},
   pages={63--82},
   issn={0020-9910},
   review={\MR{862712 (88a:58204)}},
}

\bib{Cordes}{article}{
   label={Co},
   author={Cordes, Heinz Otto},
   title={\"Uber die erste Randwertaufgabe bei quasilinearen
   Differentialgleichungen zweiter Ordnung in mehr als zwei Variablen},
   language={German},
   journal={Math. Ann.},
   volume={131},
   date={1956},
   pages={278--312},
   issn={0025-5831},
   review={\MR{0091400 (19,961e)}},
}

\bib{Evans}{article}{
   author={Evans, Lawrence C.},
   title={Classical solutions of fully nonlinear, convex, second-order
   elliptic equations},
   journal={Comm. Pure Appl. Math.},
   volume={35},
   date={1982},
   number={3},
   pages={333--363},
   issn={0010-3640},
   review={\MR{649348 (83g:35038)}},
}

\bib{Glaeser}{article}{
   author={Glaeser, Georges},
   title={Fonctions compos\'ees diff\'erentiables},
   language={French},
   journal={Ann. of Math. (2)},
   volume={77},
   date={1963},
   pages={193--209},
   issn={0003-486X},
   review={\MR{0143058 (26 \#624)}},
}


\bib{HuiskenMCF}{article}{
   author={Huisken, Gerhard},
   title={Flow by mean curvature of convex surfaces into spheres},
   journal={J. Differential Geom.},
   volume={20},
   date={1984},
   number={1},
   pages={237--266},
   issn={0022-040X},  
   review={\MR{772132 (86j:53097)}},  
}

\bib{HuiskenVol}{article}{
   author={Huisken, Gerhard},
   title={The volume preserving mean curvature flow},
   journal={J. reine angew. Math.},
   volume={382},
   date={1987},
   pages={35--48},
   issn={0075-4102}, 
   review={\MR{921165 (89d:53015)}},
}

\bib{Krylov}{article}{
   author={Krylov, N. V.},
   title={Boundedly inhomogeneous elliptic and parabolic equations},
   language={Russian},
   journal={Izv. Akad. Nauk SSSR Ser. Mat.},
   volume={46},
   date={1982},
   number={3},
   pages={487--523, 670},
   issn={0373-2436},
   review={\MR{661144 (84a:35091)}},
}

\bib{KS}{article}{
   author={Krylov, N. V.},
   author={Safonov, M. V.},
   title={A property of the solutions of parabolic equations with measurable
   coefficients},
   language={Russian},
   journal={Izv. Akad. Nauk SSSR Ser. Mat.},
   volume={44},
   date={1980},
   number={1},
   pages={161--175, 239},
   issn={0373-2436},
   review={\MR{563790 (83c:35059)}},
}

\bib{Lieb:book}{book}{
   author={Lieberman, Gary M.},
   title={Second order parabolic differential equations},
   publisher={World Scientific Publishing Co. Inc.},
   place={River Edge, NJ},
   date={1996},
   pages={xii+439},
   isbn={981-02-2883-X},
   review={\MR{1465184 (98k:35003)}},
}

\bib{Nirenberg}{article}{
   author={Nirenberg, L.},
   title={On a generalization of quasi-conformal mappings and its
   application to elliptic partial differential equations},
   conference={
      title={Contributions to the theory of partial differential equations},
   },
   book={
      series={Annals of Mathematics Studies, no. 33},
      publisher={Princeton University Press},
      place={Princeton, N. J.},
   },
   date={1954},
   pages={95--100},
   review={\MR{0066532 (16,592a)}},
}

\bib{Schneider}{book}{
   author={Schneider, Rolf},
   title={Convex bodies: the Brunn-Minkowski theory},
   series={Encyclopedia of Mathematics and its Applications},
   volume={44},
   publisher={Cambridge University Press},
   place={Cambridge},
   date={1993},
   pages={xiv+490},
   isbn={0-521-35220-7},
   review={\MR{1216521 (94d:52007)}},
}

\bib{Schnuerer}{article}{
    author={Schn{\"u}rer, Oliver C.},
   title={Surfaces contracting with speed $\vert A\vert \sp 2$},
   journal={J. Differential Geom.},
   volume={71},
   date={2005},
   number={3},
   pages={347--363},
   issn={0022-040X},
   review={\MR{2198805 (2006i:53099)}},
}
    
\bib{Schulze}{article}{
   author={Schulze, Felix},
   title={Convexity estimates for flows by powers of the mean curvature},
   journal={Ann. Sc. Norm. Super. Pisa Cl. Sci. (5)},
   volume={5},
   date={2006},
   number={2},
   pages={261--277},
   issn={0391-173X},
   review={\MR{2244700 (2007b:53138)}},
}

\bib{Smoczyk}{article}{
   author={Smoczyk, Knut},
   title={Starshaped hypersurfaces and the mean curvature flow},
   journal={Manuscripta Math.},
   volume={95},
   date={1998},
   number={2},
   pages={225--236},
   issn={0025-2611},
   review={\MR{1603325 (99c:53033)}},
}

\bib{Tso}{article}{
   author={Tso, Kaising},
   title={Deforming a hypersurface by its Gauss-Kronecker curvature},
   journal={Comm. Pure Appl. Math.},
   volume={38},
   date={1985},
   number={6},
   pages={867--882},
   issn={0010-3640},
   review={\MR{812353 (87e:53009)}},
}

\end{biblist}
\end{bibdiv}
\end{document}